\documentclass[a4paper,12pt, onesided]{amsart}

\usepackage[latin1]{inputenc}
\usepackage{amsfonts,amsmath,amssymb,amsthm}
\usepackage[english]{babel}
\usepackage{color}
\usepackage[margin = 3cm, marginparwidth = 2.4cm]{geometry}
\usepackage{pinlabel}

\newtheorem{thm}{Theorem}
\newtheorem{lemma}[thm]{Lemma}

\newtheorem{cor}[thm]{Corollary}

\theoremstyle{remark}

\newtheorem{remark}[thm]{Remark}

\numberwithin{figure}{section}
\numberwithin{equation}{section}

\newcommand{\CP}{\mathbb{CP}\vphantom{\mathbb{P}}^2}

\title{On line arrangements with odd multiplicities}
\author{Marco Golla}
\email{marco.golla@univ-nantes.fr}
\address{CNRS and Laboratoire de Math{\'e}matiques Jean Leray, Nantes Universit\'e, France}
\date{}

\begin{document}

\begin{abstract}
We give restrictions on the weak combinatorics of line arrangements with singular points of odd multiplicity using topological arguments on locally-flat spheres in 4-manifolds. As a corollary, we show that there is no line arrangement comprising 13 lines and with only triple points.
\end{abstract}

\maketitle

A classical result due to Gallai asserts that any non-trivial real line arrangement must have at least a double point. This fact, known as the Sylvester--Gallai theorem, was strengthened and generalised to pseudoline arrangements by Melchior~\cite{Melchior}. By contrast, complex line arrangements \emph{can} lack double points: for instance, the arrangement defined by the polynomial $(x^3-y^3)(y^3-z^3)(z^3-x^3)$, known as the \emph{dual Hesse arrangement}, comprises 9 lines meeting at 12 triple points. 
Using the Bogomolov--Miyaoka--Yau inequality, Hirzebruch has proven that every non-trivial complex line arrangement must contain at least a double or a triple point~\cite{Hirzebruch}. It has been speculated that the pencil of degree $3$ and the dual Hesse arrangement are the only complex line arrangements to have only triple points~\cite[Question~3]{Urzua}.

\begin{thm}\label{t:main}
Let $L$ be a locally-flat line arrangement of degree $d$ with only triple points. Then $d \equiv 1, 3, 9, 19 \pmod{24}$.
\end{thm}

Here by \emph{locally-flat} we mean that each line in the arrangement is a locally-flatly embedded sphere of self-intersection 1, and that every two lines intersect transversely and positively exactly once.

If an arrangement of degree $d$ has only triple points, then $d\equiv 1,3 \pmod 6$ by an easy counting argument. 
Therefore, Theorem~\ref{t:main} covers half of the possible degrees in which these line arrangements can exist, including in particular $d=13$ (which was the smallest previously unknown case).
The case $d=7$ is the Fano plane: that it is not realised as a complex line arrangements requires an easy elementary argument; see~\cite{RubermanStarkston} for the case of locally-flat line arrangements.

\begin{cor}
There is no pseudoline arrangement, symplectic or complex line arrangement of degree $13$ with only triple points.\hfill $\qed$
\end{cor}

Recently, the corollary has been proved for complex line arrangements by K\"uhne, Szemberg, and Tutaj-Gasi\'nska~\cite{KuhneSzembergTG}. They used the enumeration of all possible matroids with the combinatorial type of a line arrangement of 13 lines and 26 triple points, and showed that none of them is realised over \emph{any field}.

Theorem~\ref{t:main} is a direct consequence of the following more general statement.

\begin{thm}\label{t:odd}
Let $L \subset \CP$ be a locally-flat line arrangement of total degree $d$ whose singular points have only odd multiplicities. Call $t_m$ the number of singular points of $L$ of multiplicity $m$.
Then
\[
\sum (m-1)t_m \equiv d-1 \pmod{16}.
\]
\end{thm}

\subsection*{Notation and background} Homology is taken with integer coefficients. The intersection product and the signature of a 4-manifold are denoted by $\cdot$ and $\sigma$, respectively. We refer to~\cite{GS} or~\cite{Kirby} for background on 4-dimensional topology.

\vskip 0,3cm

If $X$ is a blow-up of $\CP$ at distinct points, then we have a preferred basis of $H_2(X)$ given by the homology classes of a complex line and of the exceptional divisors. We say that a class in $H_2(X)$ is \emph{characteristic} if all of its coefficients are odd in this basis. Equivalently, a class $A \in H_2(X)$ is characteristic if $A\cdot B \equiv B\cdot B$ for each $B \in H_2(X)$.

\begin{lemma}\label{l:KM}
Let $F_1, \dots, F_d \subset X$ be a collection of disjoint, locally-flat spheres in a blow-up $X$ of $\CP$ at $t$ points, such that $\sum_k [F_k] \subset H_2(X)$ is a characteristic class. Then
\[
\sum_k F_k \cdot F_k \equiv 1-t \pmod{16}.
\]
\end{lemma}

\begin{proof}
Tube the spheres $F_1, \dots, F_d$ together, to obtain a locally-flat sphere $F$ with
\[
[F] = \sum_k [F_k] \in H_2(X)
\]
and
\[
F\cdot F = \sum_k F_k \cdot F_k.
\]
Now, $F$ is characteristic and, since $X$ is a smooth manifold, its Kirby--Siebenmann invariant vanishes. Since $X$ is simply-connected, by the locally-flat version the Kervaire--Milnor theorem~\cite{KervaireMilnor}, due to Lee and Wilczy{\'n}ski~\cite{LeeWilczynski}, we have
\[
\sum_k F_k \cdot F_k = F\cdot F \equiv \sigma(X) = 1-t \pmod{16},
\]
which proves the statement.
\end{proof}

\begin{proof}[Proof of Theorem~\ref{t:odd}]
Suppose that $L = L_1 \cup \dots \cup L_d$ is a locally-flat realisation of an arrangement as in the statement. Since the multiplicities of all its singular points are odd, each line intersects an even number of lines, so $d$ is odd. Call $p_k$ the number of singular points on the line $L_k$.

Blow up $\CP$ at all the singular points of $L$, to obtain a configuration of surfaces in a blow-up $X$ of $\CP$ at $t := t_3+t_5+\dots+t_{d-2}+t_d$ points.

First, we observe that the strict transform of $L$ comprises $d$ pairwise disjoint locally-flat spheres, $F_1, \dots, F_d$, of self-intersection $1-p_1, \dots, 1-p_d$.

Since we have only blown up at points of odd multiplicity, the sum of the homology classes $[F_1],\dots,[F_d]$ is characteristic. By Lemma~\ref{l:KM}, we have:
\[
d - \sum_k p_k = \sum_k (1-p_k) \equiv 1-t \pmod{16}.
\]
Recall that $t = \sum_m t_m$ and note that
\[
\sum_k p_k = \sum_m m t_m,
\]
so that
\[
\sum_m (m-1)t_m = \sum_m m t_m - t =  \sum_k p_k - t \equiv d-1 \pmod{16}.\qedhere
\]
\end{proof}

\begin{remark}
Theorem~\ref{t:odd} generalises to configurations of rational curves whose singularities have only odd multiplicities. It also extends to certain classes of collections of spheres with conical singularities that are modelled over plane curve singularities whose multiplicity sequence only has odd entries. In turn, this generalises an obstruction for rational cuspidal curves described in~\cite[Proposition~4.2.4]{GKutle}.
\end{remark}

\begin{proof}[Proof of Theorem~\ref{t:main}]
If $L$ is a locally-flat realisation of a line arrangement of degree $d$ with only triple points, $t_3 = \frac{d(d-1)}6$, which implies that $d \equiv 0,1 \pmod 3$. To prove the theorem, it suffices to show that $d \equiv 1,3 \pmod 8$.

Applying Theorem~\ref{t:odd} we obtain:
\[
2\cdot \frac{d(d-1)}6 = (3-1)\cdot t_3 \equiv d-1 \pmod{16},
\]
which is equivalent to asking that $(d-1)(d-3) \equiv 0 \pmod{16},$ which in turn implies $d\equiv 1,3 \pmod 8$.
\end{proof}

\subsection*{Acknowledgements} I would like to thank the participants of the workshop \emph{Complex and symplectic curve configurations} for drawing my interest to this question, and especially Piotr Pokora for several stimulating conversations and for his comments on an earlier draft. I would also like to thank Paolo Aceto, Erwan Brugall\'e, Lukas K\"uhne, Danny Ruberman, and Giancarlo Urz\'ua for their input.

{\small
\bibliography{triplepoints}
\bibliographystyle{amsalpha}\itemsep-2pt
}

\end{document}